\documentclass[a4paper,11pt]{amsart}
\usepackage{amssymb}
\usepackage[dvips]{graphicx}
\usepackage{amscd}
\newcommand{\comm}[1]{}

\def\dist{\operatorname{dist}}

\def\({\left(}
\def\){\right)}

\def\raw{\rightarrow}

\def\no={\neq}

\def\B{{\mathbb B}}
\def\C{{\mathbb C}}

\def\P{{\mathbb P}}

\def\NN{{\mathcal N}}

\def\de{\delta}

\def\vep{\varepsilon}

\def\om{\omega}

\def\La{\Lambda}

\theoremstyle{plain}
\newtheorem{Main}{Theorem}

\newtheorem{Thm}{Theorem}[section]

\newtheorem{Lem}[Thm]{Lemma}

\theoremstyle{remark}

\newtheorem{Def}[Thm]{Definition}

\begin{document}
\begin{center}
\end{center}
\title[Rational Misiurewicz maps for which $J(f) \neq \hat{\C}$]
{Rational Misiurewicz maps for which the Julia set is not the whole sphere}
\author{Magnus Aspenberg}
\address{Mathematisches Seminar, Christian-Albrechts Univertit\"at zu
  Kiel, Ludewig-Meyn Str.4, 24 098 Kiel, Germany}
\email{aspenberg@math.uni-kiel.de}
\thanks{The author gratefully acknowledges funding from the Research
  Training Network CODY of the European Commission}
\begin{abstract}
We show that Misiurewicz maps for which the Julia set is not the whole
sphere are Lebesgue density points of hyperbolic maps.
\end{abstract}

\maketitle

\section{Introduction}

In \cite{RL2} by Rivera-Letelier, it is
shown that Misiurewicz maps for unicritical polynomials of the form
$f_c(z)=z^d+c$, $c \in \C$, are Lebesgue density points of hyperbolic maps.
This paper extends this result to all Misiurewicz maps in the
space of rational functions of a given degree $d \geq 2$, if the Julia
set is not the whole sphere. i.e. every Misiurewicz map for which
$J(f) \neq \hat{\C}$ is a Lebesgue density point of hyperbolic maps.
The statement is false if
the Julia set is the whole sphere (see e.g. \cite{MA}), because in
this case the Misiurewicz maps are Lebesgue density points of Collet-Eckmann
maps (CE). In addition, these CE-maps have their Julia set
equal to the whole sphere (see also \cite{MR}).

This paper complements \cite{MA4}, where Misiurewicz maps for which
$J(f)=\hat{\C}$ are studied. In particular, it is shown in that paper
that every such Misiurewicz map apart from flexible Latt\'es maps can be
approximated by a hyperbolic map. We get the following measure theoretic characterisation: Let
$f$ be a rational Misiurewicz map. Then if $f$ is not a flexible
Latt\'es map, there is a hyperbolic map arbitrarily close to
$f$. Moreover,

\begin{itemize}
\item if $J(f) = \hat{\C}$, then $f$ is a Lebesgue density point of
  CE-maps,
\item if $J(f) \neq \hat{\C}$, then $f$ is a Lebesgue density point of
  hyperbolic maps.
\end{itemize}

The notion of Misiurewicz maps goes back to the famous paper
\cite{MM} by M. Misiurewicz. In that paper, real maps of an interval are
considered and in the complex case there are some variations of the
definition of Misiurewicz maps (see e.g. \cite{GKS}, \cite{SvS}).
We proceed with the following definition.
First, let $J(f)$ be the Julia set of the function $f$ and $F(f)$ its
Fatou set. The set of critical points is denoted by $Crit(f)$ and
the omega limit set of $x$ is denoted by $\om(x)$.

\begin{Def}
A rational non-hyperbolic map $f$ is a {\em Misiurewicz map} if $f$
has no parabolic periodic points and for
every $c \in Crit(f)$ we have $\om(c) \cap Crit(f) = \emptyset$.
\end{Def}


\begin{Main} \label{mainthm}
If $f$ is a rational Misiurewicz map of degree $d \geq 2$,
for which $J(f) \neq \hat{\C}$, then $f$ is a Lebesgue density point
of hyperbolic maps in the space of rational maps of degree $d$.
\end{Main}

The space of rational maps of degree $d$ is a complex manifold of dimension
$2d+1$. To prove Theorem \ref{mainthm} we will consider
$1$-dimensional balls around the starting map $f$.
If $B(0,r)$ is a $1$-dimensional ball in the parameter space of
rational maps of degree $d \geq 2$, then we can associate a direction
vector $v \in \P(\C^{2d})$ to $B(0,r)$, such that the plane in which $B(0,r)$ lies
can be parameterized by $\{ t v : t \in \C  \}$. In this case we say
that $B(0,r)$ has direction $v$.

Theorem \ref{mainthm} above follows directly from the following.

\begin{Main} \label{discthm}
Let $r > 0$ and $f_a$, $a \in B(0,r)$ be a $1$-dimensional family of rational
functions of degree $d \geq 2$ and suppose that $f=f_0$ is
Misiurewicz map for which $J(f) \neq \hat{\C}$. Then for almost all
directions $v$ of $B(0,r)$, $f$ is a Lebesgue density point
of hyperbolic maps in the ball $B(0,r)$.
\end{Main}

We also note that combining \cite{RL2} with Theorem \ref{mainthm},
every Collet-Eckmann map for which the Julia set is not the whole
sphere can be approximated by a hyperbolic map. In particular, this
holds for all polynomial Collet-Eckmann maps. In view of \cite{RL2}
and \cite{MA} it seems natural that almost every Collet-Eckmann map
has its Julia set equal to the whole sphere.


\subsection*{Acknowledgements}
I am thankful to the referee for many useful remarks.
This paper was written at Mathematisches Seminar at
Christian-Albrechts Universit\"at zu Kiel.
The author gratefully acknowledges the hospitality of the department.

\section{Preliminary lemmas}

We will use the following lemmas by R. Ma\~n\'e.

\begin{Thm}[Ma\~n\'e's Theorem I] \label{mane}
Let $f:\hat{\C} \mapsto \hat{\C}$ be a rational map and $\La \subset J(f)$ a
compact invariant set not containing critical points or parabolic points. Then either $\La$ is a hyperbolic set or $\La \cap \omega(c) \neq \emptyset$
for some recurrent critical point $c$ of $f$.
\end{Thm}
\begin{Thm}[Ma\~n\'e's Theorem II] \label{mane2}
If $x \in J(f)$ is not a parabolic periodic point and does not intersect $\om(c)$ for some recurrent critical point $c$, then for every $\vep > 0$, there is a neighborhood $U$ of $x$ such that

\begin{itemize}
\item For all $n \geq 0$, every connected component of $f^{-n}(U)$ has diameter $\leq \vep$.

\item There exists $N > 0$ such that for all $n \geq 0$ and every connected component $V$ of $f^{-n}(U)$, the degree of $f^n |_V$ is $\leq N$.

\item For all $\vep_1 > 0$ there exists $n_0 > 0$, such that every connected component of $f^{-n}(U)$, with $n \geq n_0$, has diameter $\leq \vep_1$.

\end{itemize}
\end{Thm}
An alternative proof of Ma\~n\'e's Theorem can also be found by L. Tan and M. Shishikura in \cite{ST}.
Let us also note that a corollary of Ma\~n\'e's Theorem II is that a Misiurewicz map cannot have any Siegel disks, Herman rings or Cremer points (see \cite{RM} or \cite{ST}).

For $k \geq 0$, define
\[
P^k(f) = \overline{ \bigcup_{n > k, c \in Crit(f) \cap J(f))} f^n(c)}.
\]

Given a Misiurewicz map $f$, there is some $k \geq 0$ such that
$P^k(f)$ is a compact, forward invariant subset of the Julia set which
contains no critical points.

By Ma\~n\'e's Theorem I, the set $\La = P^k(f)$ is hyperbolic.
It is then well-known that there is a holomorphic motion $h$ on $\La$:
\[
h: \La \times B(0,r) \raw \C.
\]
For each fixed $a \in B(0,r)$ the map $h=h(z,a)=h_a$ is an injection from
$\La$ to $h_a(\La)=\La_a$ and for fixed $z \in \La$ the map $h=h(z,a)$ is
holomorphic in $a$.

Each critical point $c_j \in J(f)$ moves holomorphically, if it is
non-degenerate (i.e. $c_j$ is simple), by the Implicit Function Theorem.
If it is degenerate, we have to use a new parameterisation to be able
to view each critical point as an analytic function of the
parameters. If the parameter space is $1$-dimensional one can
use the Puiseaux expansion (see e.g. \cite{BK} Theorem 1 p. 386). By
reparameterising using a simple base change of the form
$a \raw a^q$ for some integer $q \geq 1$, the critical points then
move holomorphically. In the multi-dimensional case, i.e. if we
consider the whole $2d-2$-dimensional ball $\B(0,r)$ in the parameter
space, a corresponding result is outlined
in \cite{MA4}. Here we restrict ourselves to just state the result (it is a complex analytic version of Lemma 9.4 in
\cite{Muller-Ricci}). There is a proper, holomorphic map $\psi: U \raw V$, where $U$
and $V$ are open sets in $\C^{2d-2}$ containing the origin, such that
$f'(z,a)$ can be written as
\[
f'(z,\psi(a)) = E (z-c_1(a)) \cdot \ldots \cdot (z-c_{2d-2}(a)),
\]
where each $c_j(a)$ is a holomorphic function on $U$
and $E$ is holomorphic and non-vanishing. We therefore assume that all
critical points $c_j$ on the Julia set moves holomorphically.

We know that for some $k \geq 0$ we have $v_j:= f^{k+1}(c_j) \in \La$ for
all $c_j \in Crit(f) \cap J(f)$. Thus we can define the parameter
functions
\[
x_j(a) = v_j(a) - h_a(v_j(0)).
\]

Let $\B(0,r)$ be a full dimensional ball in the parameter space of rational maps
around $f=f_0$.
Since we already know that Misiurewicz maps cannot carry an invariant
line field on its Julia set, (see \cite{MA3}), not all the functions $x_j$ can be
identically equal to zero in $\B(0,r)$.

\begin{Lem}
If $f$ is a Misiurewicz map then at least one $x_j$ is not identically equal
to zero in $\B(0,r)$.
\end{Lem}

In fact, it follows a
posteriori, that every such $x_j$ is not identically zero. However, let us now
assume that $I$ is the set of indices $j$ such that $x_j$ is
not identically zero in $\B(0,r)$. We know that $I \neq \emptyset$. In
the end, we prove that in fact $I = \{1,\ldots,2d-2\}$.

Hence the sets $\{a: x_j(a) = 0 \}$, $j \in I$, are all analytic sets of
codimension $1$. Hence for almost all directions $v$ the funtions
$x_j$, $j \in I$ are not identically equal to zero in the corresponding disk
$B(0,r)$. From now on, fix such a disk $B(0,r)$ for some $r > 0$.

\begin{Def}
Given $0 < k < 1$, a
disk $D_0=B(a_0,r_0) \subset B(0,r)$ is a {\em $k$-Whitney disk} if
$|a_0|/r_0 = k$.

A Whitney disk is a $k$-Whitney disk for some $0 < k < 1$.
\end{Def}

We will now use a distortion lemma from \cite{MA3}, Lemma 3.5. In this
lemma we put $\xi_n=\xi_{n,j}$ and
\[
\xi_{n,j}(a) = f_a^n(c_j(a)),
\]
where $a \in B(0,r)$.
Moreover, choose some $\de' > 0$, such that $\NN$ is a fixed $10 \de'$-neighbourhood of $\La$ such that $\La_a \subset
\NN$ for all $a \in B(0,r)$ and $\dist(\La_a,\partial \NN) \geq
\de'$. This $\de' > 0$ shall be fixed throughout the paper and depends
only on $f$.

\begin{Lem} \label{initdist}
Let $\vep > 0$. If $r > 0$ is sufficiently small, there exists a number $0 <k <1$ only
depending on the function $x_j$, and a number
$S=S(\de')$, such that the following
holds for any $k$-Whitney disk $D_0=B(a_0,r_0) \subset B(0,r)$:
There is an $n > 0$ such that the set $\xi_{n}(D_0) \subset \NN$ and has diameter
at least $S$.
Moreover, we have low argument distortion, i.e.
\begin{equation}
\biggl| \frac{\xi_k'(a)}{\xi_k'(b)}  - 1 \biggr| \leq \vep,
\label{xinn}
\end{equation}
for all $a,b \in D_0$ and all $k \leq n$.
\end{Lem}

Hence, if $\vep$ is small, we have good geometry control of the shape
of $\xi_n(D_0)$ up to
the large scale $S > 0$, i.e. it is almost round. We will use the fact that this holds for
every $x_j$, $j \in I$.

\section{Conclusion and proof of Theorem \ref{discthm}}

We recall the following folklore lemma. For proofs see e.g. \cite{Nicu} (see also
\cite{JJ} for the case of polynomials).

\begin{Lem}
Let $f$ be a Misiurewicz map for which $J(f) \neq \hat{\C}$. Then the
Lebesgue measure of $J(f)$ is zero.
\end{Lem}

For each critical point $c_j=c_j(0) \in J(f)$, $j \in I$ put
$D_j=\xi_{n_j,j}(D_0)$, where $n_j$ is the number $n$ in Lemma
\ref{initdist}. Hence for every $j$, we have that the diameter of
$D_j$ is at least $S$ and we have good control of the geometry, if
$\vep > 0$ is small in Lemma \ref{initdist}.

Next we prove the following lemma.
\begin{Lem} \label{green}
For each compact subset $K \subset F(f)$ there is a perturbation
$r=r(K)$ such that $K \subset F(f_a)$ for all $a \in B(0,r)$.
\end{Lem}
\begin{proof}
It follows from \cite{ST} and \cite{RM} that the only Fatou components for Misiurewicz maps are those
corresponding to attracting cycles. Recall that $f=f_0$. 

Given $K \subset F(f_0)$, there is some integer $n$ and some small disk $B_j \subset F(f_0)$
around each attracting orbit such that $K \subset f_0^{-n}(D)$, where $D= \cup_j B_j$. Choose $D$ such that $f_0(D) \subset D$. Since $f_a(D) \subset D$ for small perturbations $a \in B(0,r)$, we have $f_a^n(D) \subset D$ for all $n \geq 0$. Hence the family $\{ f_a^n \}_{n=0}^{\infty}$ is normal on $D$ and consequently $D \subset F(f_a)$ for any such parameter $a \in B(0,r)$. Moreover, $f_a^{-n}(D)$ moves continuously with the parameter, and therefore there is some $r > 0$ such that also $K \subset f_a^{-n}(D)$ for all $ a \in B(0,r)$. The lemma is proved.
\end{proof}

Let $\de > 0$. Define
\[
E_{\de}= \{ z \in F(f_0): dist(z,J(f_0)) \geq \de \}.
\]
Now, there is some $\de_0 > 0$ (depending only on $f=f_0$) such that
for
every $0 < \de < \de_0$ there exist an $r=r(\de) > 0$ such that $E_{\de}
\subset F(f_a)$ or every $a \in B(0,r)$, by Lemma \ref{green}.

Clearly, $r(\de) \raw 0$ as $\de \raw 0$. Since the Lebesgue measure
of $J(f_0)$ is zero, for every $\vep_1 > 0$
there is some $\de > 0$ such that the Lebesgue measure of the set $\{z :
dist(z,J(f_0)) \leq \de \}$ is less than $\vep_1$. Hence we conclude
that there exists some $\de > 0$ such that for every disk $D$ of diameter at
least $S/2$ ($S > 0$ is the large scale from Lemma \ref{initdist}) we have
\[
\frac{\mu(D \cap E_{\de})}{\mu(D)} \geq 1-\vep_1.
\]

For this $\de > 0$, there is some $r=r(\de) > 0$ such that also
$E_{\de} \subset F(f_a)$, for all $a \in B(0,r)$. Since every $D_j$
contains a disk of diameter $S/2$ (because of bounded distortion), we
therefore get

\[
\frac{\mu(D_j \cap E_{\de})}{\mu(D_j)} \geq 1-\vep_1',
\]
where $\vep_1'(\vep_1) \raw 0$ as $\vep_1 \raw 0$.
By Lemma \ref{initdist},
\[
\frac{\mu(\xi_{n_j,j}^{-1}(D_j \cap E_{\de})}{\mu(D_0)} \geq 1-C \vep_1',
\]
for some constant $C > 0$ depending on the $\vep$ in Lemma
\ref{initdist}. We have $C \raw 1$ as $\vep \raw 0$. Now every parameter $a \in
\xi_{n_j,j}^{-1}(D_j \cap E_{\de})$
has that $c_j(a) \in F(f_a)$. For every parameter $a$ in the set
\[
A = \bigcap_j \xi_{n_j,j}^{-1}(D_j \cap E_{\de}),
\]
the critical point $c_j(a) \in F(f_a)$. If $I \neq \{1,\ldots,2d-2
\}$, then there is a small neighbourhood around $a$ in the ball
$\B(0,r)$ where all $c_j(a) \in F(f_a)$ for $j \in I$ and, by
asssumption (since $x_j \equiv 0$ for $j \neq I$), the other $c_j(a)$
still lands at some hyperbolic set $\La_a$. This means that $f_a$ is a
J-stable Misiurewicz map. But this contradicts \cite{MA3}. Hence $I =
\{ 1, \ldots, 2d-2 \}$, so every $x_j$ is not identically zero.

Consequently, for every $a\in A$, every $c_j(a) \in F(f_a)$ and
it follows that $f_a$ is a hyperbolic map. Since $\vep_1 > 0$ can be
chosen arbitrarly small, the Lebesgue density of hyperbolic maps at
$a=0$ is equal to $1$ and Theorem \ref{discthm} follows.

\bibliographystyle{plain}
\bibliography{ref}

\end{document}